\documentclass[10 pt,a4paper,twoside,reqno]{amsart}
\usepackage{amsfonts,amssymb,amscd,amsmath,enumerate,verbatim,calc}

\usepackage{multicol}

\def \N {{\mathbb N}}

\def \al {{\alpha}}

\def \al {{\alpha}}

\renewcommand {\textwidth}{500pt}
\renewcommand {\textheight}{650pt}
\newtheorem{theorem}{Theorem}

\newtheorem{lemma}[theorem]{Lemma}

\newtheorem{pro}[theorem]{Proposition}
\newtheorem*{orp}{Proposition 0}

\title{Number of non-primes in the set of units  modulo $n$}
\author {Abhijit A J, A. Satyanarayana Reddy}
\date{}
\begin{document}
\maketitle
\begin{abstract}
 In this work, we studied various properties of arithmetic function $\tilde{\varphi}$, where 
 $\tilde{\varphi}(n)=|\{m\in \N | 1\le m\le n, (m,n)=1, \mbox{$m$ is not a prime}\}|.$
\end{abstract}
{\bf{AMS(2010)}}: 11A05,11A25.

\section{Introduction}
For a fixed positive integer $n$, let $U_n = \{ k : 1 \le k \le n,
\gcd(k,n) = 1\}$. If $|S|$ denotes the cardinality of the set
$S$, then $|U_n| = \varphi(n)$, the well known {\em Euler-totient
function}. Let $E_n=\{m\in \N | 1\le m\le n, (m,n)=1, \mbox{$m$ is not a prime}\}.$ 
That is $E_n=\{m\in U_n| \mbox{$m$ is not a prime}\}.$ It is known that   $\varphi(n)=n-1$ if and only if $n$ is prime, hence  $E_n=\{m\in U_n| \varphi(m)\ne m-1\}.$
 Let $\tilde{\varphi}(n)=|E_n|.$ It is clear that $\tilde{\varphi}$ is an arithmetic function.  First twenty values of $\tilde{\varphi}$ are given in the following tables.\\
  \begin{tabular}{ll}
  \begin{tabular}{|l|l|l|}
 \hline
  $n$ & $E_n$  & $\tilde{\varphi}(n)$\\
  \hline
  1& $\{1\}$ & 1\\
  \hline
  2& $\{1\}$ & 1\\
  \hline
  3& $\{1\}$ & 1\\
  \hline
  4& $\{1\}$ & 1\\
  \hline
  5& $\{1,4\}$ & 2\\
   \hline
   6& $\{1\}$ & 1\\
  \hline
 7& $\{1,4,6\}$ & 3\\
  \hline
  8& $\{1\}$ & 1\\
  \hline
 9& $\{1,4,8\}$ & 3\\
  \hline
  10& $\{1,9\}$ & 2\\
  \hline
  \end{tabular} &
\begin{tabular}{|l|l|l|}
 \hline
  $n$ & $E_n$  & $\tilde{\varphi}(n)$\\
  \hline
  11& $\{1,4,6,8,9,10\}$ & 6\\
  \hline
  12& $\{1\}$ & 1\\
  \hline
  13& $\{1,4,6,8,9,10,12\}$ & 7\\
  \hline
  14& $\{1,9\}$ & 2\\
  \hline
  15& $\{1,4,8,14\}$ & 4\\
   \hline
   16& $\{1,9,15\}$ & 3\\
  \hline
 17& $\{1, 4, 6, 8, 9, 10, 12, 14, 15, 16\}$ & 10\\
  \hline
  18& $\{1\}$ & 1\\
  \hline
 19& $\{1, 4, 6, 8, 9, 10, 12, 14, 15, 16, 18\}$ & 11\\
  \hline
  20& $\{1,9\}$ & 2\\
  \hline
  \end{tabular}
\end{tabular}\\

The following results provides few immediate properties of $\tilde{\varphi}$ function. 
Recall that  for $n\in \N$, $\pi(n), \omega(n)$ denotes the number of primes less than or equal to $n$,
the number of distinct primes divisors of $n$ respectively. We denote $N_j=p_1p_2\cdots p_j$, product of first $j$ primes, for example 
$N_3=30.$ It would be interesting to note that $\frac{\sqrt{n}}{2}<\varphi(n)<n-\sqrt{n}$ and by prime number theorem, for a large enough $n$, $\pi(n)\approx\frac{n}{log(n)}$. But, $\omega(n)$ does not have any ordinary behaviour. In fact for any natural number $k$, one can find a subsequence $\{n_i\}$ of natural numbers such that $\omega(n_i)=k$ for all $i$. To know more properties of  $\pi(n), \omega(n), \varphi(n)$ and 
$U_n$ refer any one of~\cite{apostol,burton,H:W,J:J}.

\begin{orp}\label{orp:immediate}
If $p_k$ denotes the $k^{th}$ prime, then $\tilde{\varphi}(p_k)=p_k-k.$  
 In general $\tilde{\varphi}(n)=\varphi(n)-\pi(n)+\omega(n).$
\end{orp}
\begin{proof}
  First part follows  from the observation that $E_{p_k}=U_{p_k}\setminus \{p_1,p_2,\ldots,p_{k-1}\}.$\\
 Let $n=q_1^{\al_1}q_2^{\al_2}\cdots q_k^{\al_k}.$ We claim $U_n\cup \{q_1,q_2,\ldots,q_k\}=E_n\cup \{p_1,p_2,\ldots,p_{\pi(n)}\}$
 Suppose $x\in U_n\cup \{q_1,q_2,\ldots,q_k\}.$ Now if $x$ is prime, then 
 then $x\in \{p_1,p_2,\ldots,p_{\pi(n)}\}$ and if $x$ is not prime, then $x\in E_n.$ Thus $x\in E_n\cup \{p_1,p_2,\ldots,p_{\pi(n)}\}.$
 
 Suppose $y\in E_n\cup \{p_1,p_2,\ldots,p_{\pi(n)}\}.$ If $y\in E_n$, then $y\in U_n$ if $y\in  \{p_1,p_2,\ldots,p_{\pi(n)}\}$ then $y\in U_n$ or 
 $y\in \{q_1,q_2,\ldots,q_k\}$ depending on $y\nmid n$ or $y|n$ respectively. 
 By definition, $E_n\cap \{p_1,p_2,\ldots,p_{\pi(n)}\}=U_n\cap \{q_1,q_2,\ldots,q_k\}=\emptyset.$ Hence
 $\tilde{\varphi}(n)+\pi(n)=|E_n|+|\{p_1,p_2,\ldots,p_{\pi(n)}\}|=|U_n|+|\{q_1,q_2,\ldots,q_k\}|=\varphi(n)+\omega(n).$
 \end{proof}

\begin{pro}\label{pro:immediate}
\begin{enumerate}
 \item \label{pro:immediate:1} If $n_0$ is the  square free part of $n$ then $\tilde{\varphi}(n)\ge \tilde{\varphi}(n_0).$
 \item\label{pro:immediate:2}  Let $n,k\in \N$ and $n\ge 3.$ Then  $\tilde{\varphi}(n^k)< \tilde{\varphi}(n^{k+1}).$ 
 \item \label{pro:immediate:3} Let $n,p,q\in \N$ where $p,q$ primes with $p<q$ and $(p,n)=1$. Then $\tilde{\varphi}(np)\le \tilde{\varphi}(nq).$
  \item \label{pro:immediate:4}    If $a\in \N$ and $\omega(a)=i$, then $\tilde{\varphi}(N_i)\le \tilde{\varphi}(a).$
\end{enumerate}
\end{pro}
\begin{proof}
  Proof of Part~\ref{pro:immediate:1}: Since $n_0$ and $n$ have same set of prime factors, $E_{n_0}\subseteq E_n.$\\
  Proof of Part~\ref{pro:immediate:2}: It is easy to see that  $E_{n^k}\subseteq E_{n^{k+1}}.$ Sufficient to show that 
  $E_{n^{k+1}}\setminus E_{n^{k}}\ne \emptyset.$ \\ Since $n\ge 3$ 
  there exists  $\ell\in \{2,3,\ldots\}$ such that  $(n-1)^{\ell-1}\le n^k< (n-1)^{\ell}.$ Consequently $n^k<(n-1)^{\ell}<(n-1)^{\ell-1}\cdot n\le n^{k+1}.$
  Thus we produced an element $(n-1)^{\ell}$ in $E_{n^{k+1}}\setminus E_{n^{k}}.$\\  
  Proof of Part~\ref{pro:immediate:3}: Let $x\in E_{np}$  and $q^m||x$ for some $m\in \N\cup\{0\}.$ Then $x=q^m y$ for some  $y\in \N.$  Further 
that $(p^my, nq)=1$ as $(p^m,n)=1, (y,q)=1.$ Now we define a map $f:E_{np}\to E_{nq}$ as $f(x)=p^my.$ It is easy to see that 
$f$ is one-one. Hence the result follows.\\
Proof of Part~\ref{pro:immediate:4} Let $q_1,q_2,\ldots, q_i$ be all distinct  prime divisors of $a$ and $s=q_1q_2\cdots q_i.$ 
Then from the proof of Part~\ref{pro:immediate:1} we have $\tilde{\varphi}(s)\le \tilde{\varphi}(a).$ 
Hence the result follows from the facts that  $N_i\le s$ and  $E_{N_i}\setminus\{q_1,q_2,\ldots, q_i\}\subseteq E_s\setminus\{p_1,p_2,\ldots,p_i\}.$
\end{proof}
\section{Main results}
It is easy to see that $\tilde{\varphi}(n)=1$  if and only if  $n\in \{1, 2, 3, 4, 6, 8, 12, 18, 24, 30\}.$ That is $\tilde{\varphi}(\ell)>1$ whenever 
$\ell \ge 31.$ In general we prove.

\begin{theorem}\label{thm:finiteness}
 For every $n\in \N$, there exists $N\in \N$ such that  $\tilde{\varphi}(m)>n$, for all $m> N.$ 
\end{theorem}

\begin{lemma} \label{lemma:finiteness1}
 Let $n\in \N.$ There exists  an $M_n\in \N$ such that  $\tilde{\varphi}(N_k)>n$ for all $k\in \N$, $k\ge M_n.$
\end{lemma}

\begin{proof} Let $p_i$ denote the $i^{th}$ prime. 
 For $i\ge 4,$ we define $P_i=\{m\in \N| \varphi(m)=m-1,m\in U_{N_i}\}$, $Q_i=\{r\in P_i|rp_{i+1}<N_i\}.$
 From Bertrand's Postulate $p_{i+1}\in Q_i.$  Suppose $Q_i=\{p_{i+1},p_{i+2},\ldots, p_h\}$ then $|Q_i|=h-i.$ Further 
 $p_{i+2}< 2p_{i+1}$ and $p_{h+2}<2p_{h+1}<4p_h.$ Thus $p_{i+2}p_{h+2}<8p_{i+1}p_h<8N_i<N_{i+1}$ as $p_{i+1}>8.$
 Hence $\{p_{i+2},\ldots,p_{h+2}\}\subseteq Q_{i+1}$ consequently $|Q_i|<|Q_{i+1}|.$ Now $\{rp_{i+1}|r\in Q_i\}\subset E_{N_i}.$
 Thus $\tilde{\varphi}(N_i)>|Q_i|$. Hence the result follows from the fact that $|Q_i|$  increases with $i.$
\end{proof}

 Since $|Q_4|=4$, therefore for all $n\geq 4$, $|Q_n|\geq n$ and  $\tilde{\varphi}(N_k)> k$ when $k\ge 4$. Hence $M_n\le \max\{4,n\}$. Further, it is easy to see that $\tilde{\varphi}$ is an increasing function on $\{N_3, N_4, N_5, \ldots\}.$

\begin{lemma}\label{lemma:finiteness2}
 Let $a,b\in \N$ and $A(a,b)=\{n\in \N | \tilde{\varphi}(n)=a, \omega(n)=b\}.$ Then $A(a,b)$ is  finite. 
\end{lemma}
\begin{proof}
 Let $p$ be the $(b+\ell)^{th}$ prime, where $ a <\frac{\ell(\ell+1)}{2}.$ Let $n\in \N$ with $\omega(n)=b$ and $n> p^2.$
 Since  there are at least $\ell$ primes less than $p$,which are co-prime to $n$, we can choose two such primes $q_1, q_2.$  Then $q_1q_2<p^2< n.$ 
 Thus $q_1q_2\in E_n$ as a consequence $\tilde{\varphi}(n)>\frac{\ell(\ell+1)}{2}>a.$ Hence the result follows.
\end{proof}

Now we will prove Theorem~\ref{thm:finiteness}.
\begin{proof}
 Let $n\in \N.$ There  exists $k\in \N$ such that if $\omega(t)>k$, then $\tilde{\varphi}(t)\ge \tilde{\varphi}(N_{\omega(t)})> n.$
 Now from  Lemma~\ref{lemma:finiteness2} for each $i\in \N$, $i\le k$ there exists $m_i\in \N$ such that for all $t\in \N$ with 
 $\omega(t)=i$ and $t>m_i$, then $\tilde{\varphi}(t)>n.$ 
 Thus if  $N=\max\{m_1,m_2,\ldots,m_n\}$ and $\ell>N,$  then $\tilde{\varphi}(\ell)>n.$ 
\end{proof}

\section{Future work}
The following question is natural.\\ 
Let $k\in \N$ be fixed. Find all $n$ such that $\tilde{\varphi}(n)=k.$\\
One can see the following result by calculating the appropriate bounds and checking accordingly.
\begin{lemma}\label{s(k)}
Let $s(k)=\{n\in \N|\tilde{\varphi}(n)=k\}.$ Then  the set of values $k\in \{1,2,\ldots,100\}$ such that $s(k)$ is empty set is $\{13,31,70\}.$
\end{lemma}

{\bf Conjecture:} If  $M=\{\tilde{\varphi}(n)|n\in \N\},$ then $\N\setminus M$ contains infinite number of elements.\\

Also recall the Carmichael conjecture~\cite{car}: if $N(m)=|\{n\in \N|\varphi(n)=m\}|$, then $N(m)\ne 1.$ This conjecture no longer true for $\tilde{\varphi}$ as $\tilde{\varphi}(n)=16$ only for $n=144.$
If $s(k)$ is as defined in the Lemma~\ref{s(k)}, then then the values of $k\in \{1,\ldots,100\}$ such that  $s(k)$ contains a single element is $\{16, 39, 47, 49, 53, 57, 58, 65, 66, 76, 85, 91, 94\}.$

The following tables and observations may be useful in answering  the above questions.

Given a positive integer $k$, the following tables provides the smallest  value of $n$ such that 
$\tilde{\varphi}(n)=k.$ For example $\tilde{\varphi}(1)=\tilde{\varphi
}(2)=\tilde{\varphi}(3)=\tilde{\varphi}(4)=1,$ hence the smallest value $n$ such that $\tilde{\varphi}(n)=2$ is $5.$\\

\begin{tabular}{|l|l|l|l|l|l|l|l|l|l|l|l|l|l|l|l|l|}
\hline
k& 1&2&3&4&5&6&7&8&9&10&11&12&13&14&15\\
  n&1&5&7&15&26&11&13&38&102&17&19&25&??&23&35\\
 \hline
  \end{tabular}
 
 \vspace{3mm}
 
 \begin{tabular}{|l|l|l|l|l|l|l|l|l|l|l|l|l|l|l|l|l|}
 \hline
 k&16&17&18&19&20&21&22&23&24&25&26&27&28&29&30
 \\
 n&144&74&198&29&31&75&57&104&94&37&55&69&41&43&118\\
 \hline
 \end{tabular}
 
 \vspace{3mm}
 
 \begin{tabular}{|l|l|l|l|l|l|l|l|l|l|l|l|l|l|l|l|l|}
 \hline
 k&31&32&33&34&35&36&37&38&39&40&41&42&43&44&45\\
 n&??&47&81&128&87&134&53&93&480&146&77&59&61&117&111\\
    \hline
    \end{tabular}
    
    \vspace{3mm}
    
    \begin{tabular}{|l|l|l|l|l|l|l|l|l|l|l|l|l|l|l|l|l|}

  \hline
  k&46&47&48&49&50&51&52&53&54&55&56&57&58&59&60\\
  n&166&172&67&250&91&71&73&350&194&129&202&79&206&212&83\\
  \hline
 \end{tabular}\\

The following observations are useful for the question posed in the beginning  of the section for the case $k\le 20.$ 
\begin{multicols}{2}
 \begin{enumerate}
 \item $\tilde{\varphi}(n)=1\Leftrightarrow n\in \{2, 3, 4, 6, 8, 12, 18, 24, 30\}.$
 \item $\tilde{\varphi}(n)=2\Leftrightarrow n\in \{5, 10, 14, 20, 42, 60\}.$
 \item $\tilde{\varphi}(n)=3\Leftrightarrow n\in \{7, 9, 16, 36, 48, 90\}.$
 \item $\tilde{\varphi}(n)=4 \Leftrightarrow n\in \{15,22,54,84\}.$
 \item $\tilde{\varphi}(n)=5 \Leftrightarrow n\in \{26, 28, 66, 120\}.$
 \item $\tilde{\varphi}(n)=6\Leftrightarrow n\in \{11, 21, 32, 40, 72, 78, 210\}.$
 \item $\tilde{\varphi}(n)=7\Leftrightarrow n\in \{13,34,50\}.$
 \item $\tilde{\varphi}(n)=8\Leftrightarrow n\in \{38, 44, 70, 150\}.$
 \item $\tilde{\varphi}(n)=9\Leftrightarrow n\in \{102, 114, 126\}.$
 \item $\tilde{\varphi}(n)=10\Leftrightarrow n\in \{17, 27, 46, 56, 96, 108, 180\}.$
 \item $\tilde{\varphi}(n)=11\Leftrightarrow n\in \{19,33,52,132\}.$
 \item $\tilde{\varphi}(n)=12\Leftrightarrow n\in \{25,45,80,168\}.$
  \item $\tilde{\varphi}(n)=14\Leftrightarrow n\in \{23,39,58,62,110,138\}.$
   \item $\tilde{\varphi}(n)=15\Leftrightarrow n\in \{35,64,68,156,240\}.$
   \item $\tilde{\varphi}(n)=16\Leftrightarrow n=144.$
    \item $\tilde{\varphi}(n)=17\Leftrightarrow n\in \{74,76,100,140\}.$
     \item $\tilde{\varphi}(n)=18\Leftrightarrow n\in \{198,270,330\}.$
      \item $\tilde{\varphi}(n)=19\Leftrightarrow n\in \{29,51,88,98,162,174,420\}.$
       \item $\tilde{\varphi}(n)=20\Leftrightarrow n\in \{31,63,82,130\}.$
 \end{enumerate}
 \end{multicols}
 {\bf Acknowledgment:} We would  like to thank referee for valuable comments.

\begin{tabular}{ll}
Abhijit A J, & A. Satyanarayana Reddy,\\
Department of Mathematics,&Department of Mathematics,\\
Shiv Nadar University, & Shiv Nadar University,\\
UP, India-201314. & UP, India-201314.\\ 
(e-mail:aj448@snu.edu.in).& (e-mail:satyanarayana.reddy@snu.edu.in).
\end{tabular}

\begin{thebibliography}{10}
\scriptsize
 \bibitem{apostol} T. M. Apostol, {\em Introduction to Analytic Number Theory}, Springer-Verlag, 1976.

  \bibitem{burton} David M. Burton, {\em Elementary Number Theory}, McGraw-Hill Higher Education, 7th edition, 2010.
  \bibitem{car} R.D. Carmichael {\em Note on Euler's-$\varphi$ function}, Bull. Amer.Math.Soc.28 (1922) 109-110.

\bibitem{H:W} G. H. Hardy and E. M. Wright, {\em An Introduction to the Theory of 
Numbers}, 5th edition, Oxford
University Press, Oxford, 1979.
   \bibitem{J:J} Gareth A.Jones and J.Mary Jones, {\em Elementary Number Theory},Springer undergraduate
mathematics series,1998.
\end{thebibliography}
\end{document}